\pdfoutput=1 

\documentclass[justified,notoc,twoside]{tufte-handout}
\usepackage[T1]{fontenc}
\usepackage[english]{babel}
\usepackage[tracking]{microtype}
\usepackage{listings}
\usepackage[inline]{asymptote}

\usepackage[strict]{changepage}

\usepackage{amsthm}
\usepackage{thmtools}
\usepackage{thm-restate}
\usepackage{hyperref}
\usepackage{cleveref}

\usepackage{beton}
\usepackage{sectsty}
\allsectionsfont{\fontfamily{LinuxBiolinumT-OsF}\selectfont}
\newcommand\mytextbf[1]{{\textbf {\fontfamily{LinuxBiolinumT-OsF}\selectfont #1}}}
\usepackage{amsmath} 
\usepackage{amssymb}
\usepackage{units} 
\usepackage[osf,sc]{mathpazo} 

\usepackage[most]{tcolorbox}
\usepackage{tocloft}   
\cftsetindents{section}{0em}{2em}
\cftsetindents{subsection}{0em}{2em}
\let\subsubsection\subsection

\usepackage[caption=false]{subfig}
\usepackage{graphicx} 
\setkeys{Gin}{width=\linewidth,totalheight=\textheight,keepaspectratio}
\graphicspath{{figures/}} 
\usepackage{booktabs} 
\usepackage{multicol} 
\usepackage{fancyvrb} 
\fvset{fontsize=\normalsize} 

\newcommand{\bothcite}[2][0cm]{\citep{#2}\cite[#1]{#2}}

\AtBeginDocument{ \fontsize{12}{14}\selectfont } 

\title[GP I.1] 
{Graph Puzzles I.1:\\
Oriented Berge--Fulkerson Conjecture} 

\date{}

\begin{document}

\selectlanguage{english}

\maketitle 

Nikolay Ulyanov\marginnote{ulyanick@gmail.com}

\begin{abstract}

\noindent

\newthought{Abstract: }The Berge--Fulkerson conjecture states that every bridgeless cubic graph can be covered with six perfect matchings such that each edge is covered exactly twice. An equivalent reformulation is that it's possible to find a 6-cycle 4-cover. In this paper we discuss the oriented version (o6c4c) of the latter statement, pose it as a conjecture and prove it for the family of Isaacs flower snarks.

Similarly to the case of oriented cycle double cover, we can always construct an orientable surface (possibly with boundary) from an o6c4c solution. If the o6c4c solution itself splits into two (not necessarily oriented) cycle double covers, then it's also possible to build another pair of orientable surfaces (also possibly with boundaries). Finally we show how to build a ribbon graph, and for some special o6c4c cases we show that this ribbon graph corresponds to an oriented 6-cycle double cover.

\end{abstract}

\noindent\rule{5in}{0.4pt}

\begin{figure*}
  \checkoddpage \ifoddpage \forcerectofloat \else \forceversofloat \fi
  \includegraphics[width=0.93\linewidth]{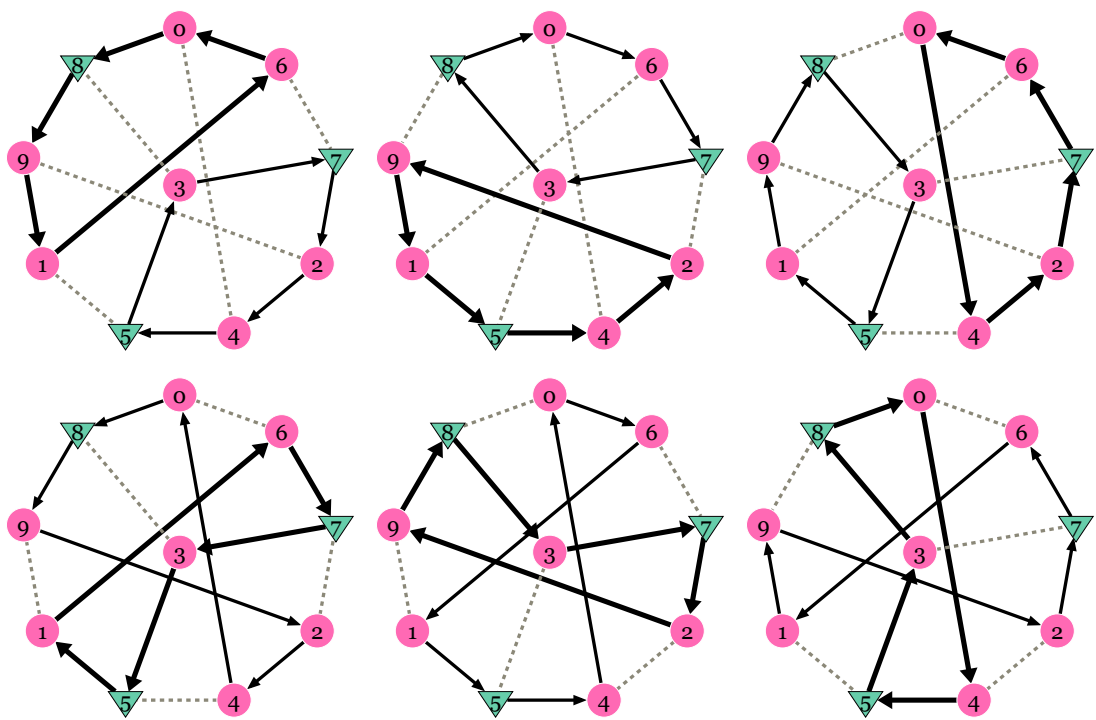}
  \caption{Solution for Petersen graph}
  \label{fig:petersen}
\end{figure*}

\pagebreak

\begin{tcolorbox}[
    breakable,
    toggle enlargement=evenpage,
    rounded corners, 
    boxrule=1pt, 
    width=1.2\linewidth, 
]
    \csname @starttoc\endcsname{toc}
\end{tcolorbox}

\section{Introduction}\label{sec:Introduction}

\subsection{tl;dr}
In this paper we would like to introduce and explore the oriented Berge--Fulkerson conjecture:

\begin{restatable}[o6c4c, oriented 6-cycle 4-cover]{conj}{obf}
Every bridgeless graph has a collection of 6 cycles, in which every cycle can be oriented in such a way that each edge is covered exactly by 4 different cycles: 2 times in one direction, and 2 times in an opposite direction.
\end{restatable}

\subsection{Motivation and related conjectures}

Let $G$ be a graph, with vertex set $V(G)$, and edge set $E(G)$. In this paper we assume that all graphs are finite, and without loops\sidenote[][]{Edges which start and end on the same vertex.}. For terminology and notation you may consult with any standard textbook, e.~g. \bothcite{Zhang2012}.

A \textit{cubic} (or \textit{3-regular}, or \textit{trivalent}) graph is one where each vertex has degree 3.

An edge is called a \textit{bridge} if its removal disconnects the graph. If the graph has no bridges, then it is said be \textit{bridgeless}.

A \textit{perfect matching} is a set of edges, such that every vertex in $G$ is incident with exactly one edge in this set. The following classical conjecture is due to C.~Berge and D.~R.~Fulkerson \bothcite{Fulk71}.

\begin{restatable}[Berge, Fulkerson]{conj}{}
Every bridgeless cubic graph has a collection of 6 perfect matchings, which together cover all edges exactly twice.
\end{restatable}

Note that the complement of a perfect matching in a cubic graph is a 2-factor (a subgraph, where each vertex has degree 2). 2-factor is also a \textit{cycle}. This term is usually used in this part of graph theory to denote \textit{even subgraphs}, subgraphs where each vertex has even degree (if a cycle is connected, it is also said to be a \textit{circuit}). The Berge--Fulkerson conjecture is equivalent to the following statement:

\begin{restatable}[6c4c]{conj}{}
Every bridgeless cubic graph has a collection of 6 cycles, such that each edge is covered exactly 4 times.
\end{restatable}

Note that for cubic graphs 6 would be the minimal number of cycles required to cover all edges 4 times. 

With the reformulation above (which doesn't rely on perfect matchings, but only on cycles), it's also possible to consider non-cubic graphs. For them a similar statement is also conjectured (which is similar in strength, because we can reduce non-cubic graphs to cubic ones):

\begin{restatable}[6c4c for non-cubic graphs]{conj}{}
Every bridgeless graph has a collection of 6 cycles, such that each edge is covered exactly 4 times.
\end{restatable}

J.~C.~Bermond, B.~Jackson and F.~Jaeger \bothcite[-5cm]{BJJ} has proved that every bridgeless graph has a collection of 7 cycles, that cover each edge exactly 4 times.

There is another classical conjecture, known as \textit{cycle double cover conjecture}, attributed to W.~T.~Tutte\sidenote[][-3.5cm]{No paper reference here.}, A.~Itai and M.~Rodeh \bothcite[-2.5cm]{ItaiRodeh}, G.~Szekeres \bothcite{Szekeres1973} and P.~Seymour \bothcite{Seymour1979} (but also see the book by C.-Q.~Zhang \bothcite{Zhang2012} for some more details about its origin):

\begin{restatable}[cdc]{conj}{}
Every bridgeless (not necessarily cubic) graph has a collection of cycles, such that each edge is covered exactly 2 times.
\end{restatable}

It's also possible to consider cycles with \textit{orientations}. The idea would be then to cover each edge in opposite directions. The following conjecture (oriented/orientable/directed cycle double cover) is due to F.~Jaeger \bothcite{Jaeger1985Survey}:

\begin{restatable}[ocdc]{conj}{}
Every bridgeless graph has a cycle double cover, in which every cycle can be oriented in such a way that each edge is covered exactly by 2 different cycles in 2 different directions.
\end{restatable}

In the recent years Berge--Fulkerson conjecture has been proven for several families of snarks. Nevertheless, to our knowledge, the oriented version of Berge--Fulkerson conjecture hasn't been considered in the literature before. The closest we could find is a page on the Open Problem Garden website, posted by M.~DeVos \bothcite{OPGBF}\bothcite{OPGMNCC}, which mentions a weaker version with 8 cycles. However, we think that 6 cycles is enough for bridgeless graphs, even in the oriented version:

\obf*

\subsection{Main contribution}

In this paper:
\begin{itemize}
\item We explore the notion of oriented 6-cycle 4-cover with a number of examples;
\item Prove existence of oriented 6-cycle 4-cover for an infinite family of Isaacs flower snarks;
\item Show how to glue the covering into an orientable surface, possibly with boundary;
\item In cases where the oriented 6-cycle 4-cover could be split into two 6-cycle double covers, we show a different construction of a pair of orientable surfaces, also possibly with boundaries;
\item In some other special cases, where all vertices are \textit{"disordered"} (we will define this later in the paper), we show how to build an oriented 6-cycle double cover;
\item And discuss some miscellaneous stuff around the topic.
\end{itemize}

\subsection{Note about the paper series title}

This is the first paper in the series, where we explore various graph decomposition conjectures, with the aid of computations, and additionally releasing all of the code as open source. The current plan is to cover a wide range of conjectures in graph theory, starting from famous conjectures about snarks, and later exploring other graph decomposition areas, such as:

\begin{itemize}
  \item Conjectures related to snarks (Cycle double cover conjecture, Berge--Fulkerson conjecture, Petersen colouring conjecture, Unit vector flows, etc.);
  \item Graph labeling conjectures (Kotzig--Ringel--Rosa conjecture aka Graceful Tree conjecture, harmonious labeling conjecture, edge-graceful labeling conjecture);
  \item Various versions of Graham--Häggkvist conjecture about decompositions of regular graphs into trees;
  \item Tree packing conjecture.
\end{itemize}

The original motivation behind this project has been not to prove any of the conjectures in particular, but to explore the algorithms for constructing the solutions, and trying to push the limits of the statements of the conjectures, until they eventually break. The project started around 2013, and originally it was intended to have an "experimental mathematics" flavour, but now it also includes rigorous mathematical results.

\section{o6c4c panorama}

\subsection{o6c4c for 3-edge-colourable cubic graphs}
As is well-known (e.~g. by using Vizing's theorem), the cubic graphs can be divided into 2 classes. Class~1 cubic graphs have edge chromatic number 3, and they are also known as \textit{3-edge-colourable} graphs. Class~2 cubic graphs have edge chromatic number 4, and will be considered in further sections.

Let's consider a 3-edge-colourable graph. We can form 3 cycles, each containing edges from 2 distinct colours, so that:
\begin{itemize}
  \item First cycle includes edges of colours 1 and 2;
  \item Second cycle includes edges of colours 1 and 3;
  \item Third cycle includes edges of colours 2 and 3.
 \end{itemize}

It's easy to see that this set of cycles produces a 3-cycle double cover, or a 6-cycle 4-cover, if we duplicate the cycles. Now, if we take these 3 cycles, give them some arbitrary orientation, and duplicate them with same cycles with opposite orientation, we will produce a solution to the oriented Berge--Fulkerson conjecture.

\subsection{o6c4c for Petersen graph}
\label{pet_o6c4c}

As it usually turns out, the more interesting class of graphs to consider is the non-3-edge-colourable cubic graphs. Petersen graph is the smallest and presumably the most famous class 2 cubic graph.\sidenote[][-0.5cm]{Technically it's a \textit{snark}, see next subsection.}

Berge--Fulkerson solution for the Petersen graph consists of all 6 distinct perfect matchings, or equivalently, of all 6 distinct 2-factors. An oriented solution is visualized in Fig.~\ref{fig:petersen} on title page. All other o6c4c solutions for Petersen graph are isomorphic to the depicted solution.

We would like to note several interesting properties of this solution:

\begin{itemize}
  \item There are 3~vertices (with indices 5,~7~and~8 in the Fig.~\ref{fig:petersen}) that look different locally than the other 7~vertices. We will denote them as \textit{ordered} vertices, while the others will be called \textit{disordered} vertices, see Fig.~\ref{fig:vertices}. The difference is that ordered vertices have same pairs of edges oriented the same way, while disordered vertices have opposite orientations for same pairs of edges.

\begin{figure}
  \checkoddpage \ifoddpage \forcerectofloat \else \forceversofloat \fi
  \centering
  \subfloat[Ordered vertex]{\includegraphics[width=0.5\linewidth]{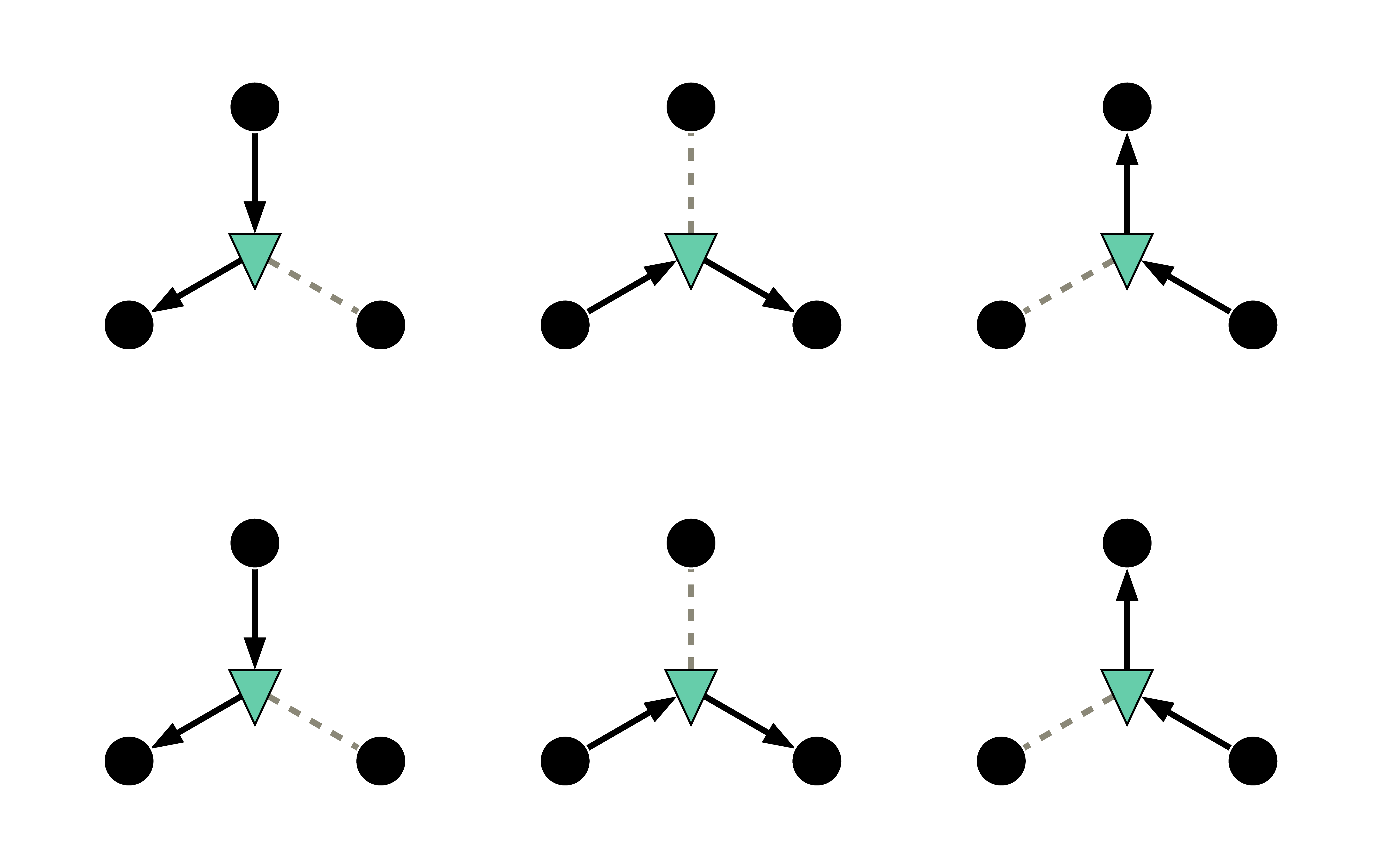}}
  \subfloat[Disordered vertex]{\includegraphics[width=0.5\linewidth]{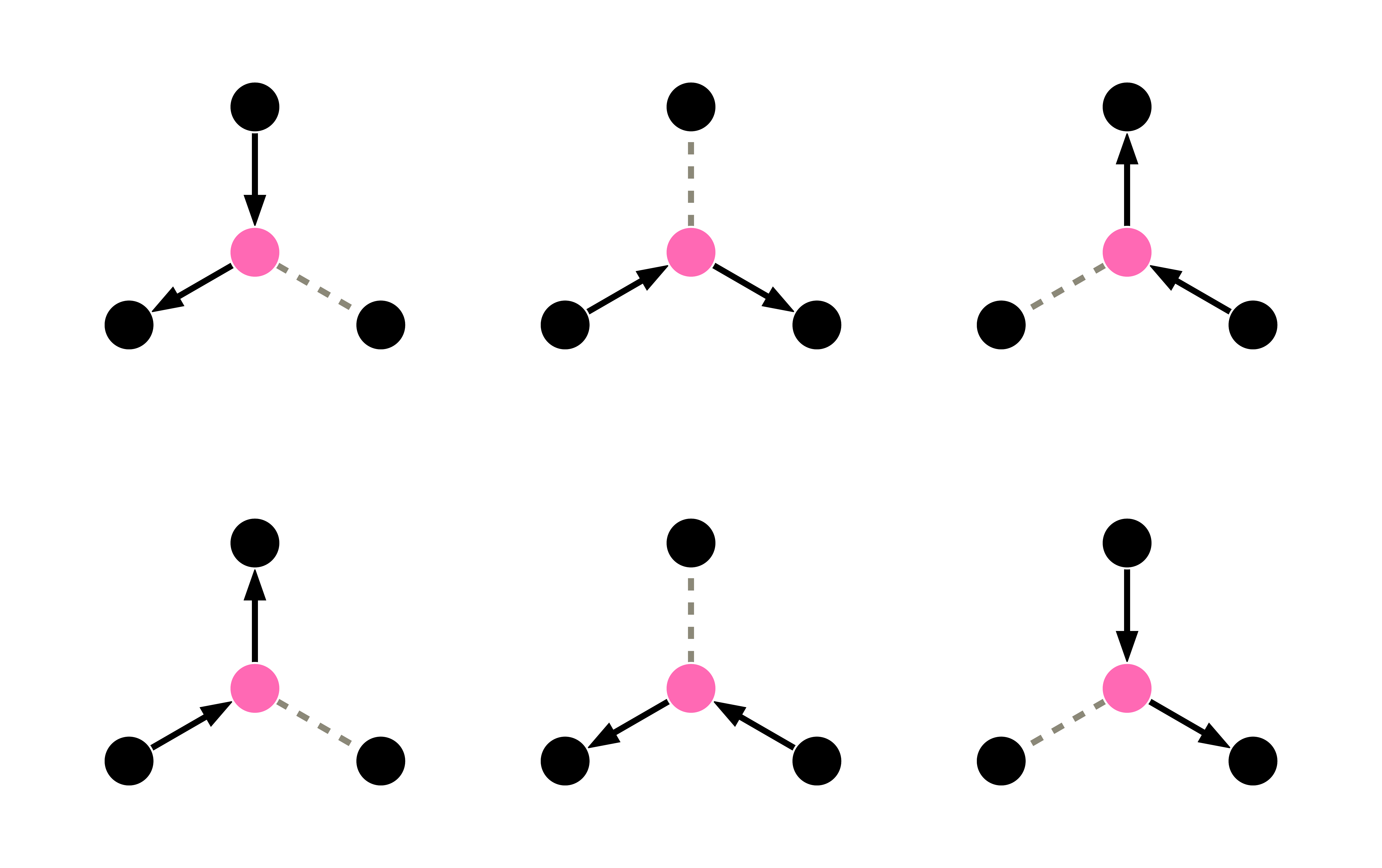}}
  \caption{2 types of vertices}
  \label{fig:vertices}
  \setfloatalignment{b}
\end{figure}

  \item The same solution has actually already appeared in the literature several times, although in different guises, which we will discuss in a later section.
  
  \item This o6c4c solution has an additional property that the highlighted bold subset of cycles forms a 6-cycle double cover (so the whole o6c4c actually splits into two 6-cycle double covers), we will return to this case again in a later section about the pair of orientable surfaces.
\end{itemize}

\subsection{o6c4c for small snarks}

There's a special subset of class 2 cubic graphs called \textit{snarks}. Exact definition is not super important for this paper, but for precision of the next sentence, we clarify that snarks here have girth $\ge 5$ and cycle connectivity $\ge 4$.

We have computationally verified, that o6c4c solutions exist for all snarks with 30 vertices or less, by brute-force search on the database shared by House of Graphs \citep{HoG}\sidenote[][-1.5cm]{\url{https://houseofgraphs.org/meta-directory/snarks}}. The code is released in the author's GitHub\sidenote[][-0.5cm]{\url{https://github.com/gexahedron/cycle-double-covers}}.

\subsection{o6c4c for Isaacs flower snarks, and almost strong Petersen colouring}

In this section we will show a construction of o6c4c solutions for flower snarks, based on the \bothcite{HagglundSteffen} paper, in which the authors explore the notion of Petersen colouring (which implies both 6c4c and 5cdc). They also explore the notion of \textit{strong Petersen colouring}, which was introduced in \bothcite{Jaeger1985PC}; when it exists for the graph it implies both o6c4c and o5cdc solutions. However, it seems that very rare snarks have this property (e. g., only 8 among 3247 snarks with number of vertices $\le 28$ have this colouring), and the authors also proved that flower snarks don't have this property. The authors also prove that flower snarks have non-strong Petersen colouring using $K_{3,3}^*$-reduction technique.

Flower snarks were discovered by R.~Isaacs in \bothcite{Isaacs}. For odd $k \ge 1$ the flower snark $J_{2k+1}$ has vertex set
$$V\left(J_{2k+1}\right) = \left\{a_i, b_i, c_i, d_i \middle| i = 0, 1,\ldots, 2k\right\},$$
and edge set
$$E\left(J_{2k+1}\right) = \left\{a_i b_i, b_i c_i, b_i d_i, a_i a_{i+1}, c_i d_{i+1}, d_i c_{i+1} \middle| i = 0, 1,\ldots, 2k\right\},$$
where indices are added modulo $2k+1$.\sidenote[][]{For $k=1$ it's a weak snark with girth $3$, also known as Tietze's graph.}

\begin{restatable}[o6c4c for flower snarks]{thm}{}
For every $k \ge 1$ the flower snark $J_{2k+1}$ has an o6c4c solution, where all $b_i$ vertices are ordered.
\end{restatable}
\begin{proof}
The construction of o6c4c solution also follows quite straightforwardly from the $K_{3,3}^*$-reduction, all we need is to add the orientations appropriately. Because flower snark $J_{2k+3}$ is $K_{3,3}^*$-reducible to $J_{2k+1}$ (see Lemma 4.2 of \bothcite{Steffen}), we only need to consider $J_3$ and $J_5$.

$J_3$ is just a Petersen graph, with a vertex replaced with a triangle, the solution basically stays the same. As for $J_5$ we track down, where each of the vertices is mapped to, and then rebuild the o6c4c solution back from $J_3$ to $J_5$. See figures~\ref{fig:j3vsj5part1} and~\ref{fig:j3vsj5part2} for solution comparison. Note, that all of the edges that are neighbours of $a_4, b_4, c_4, d_4, a_5, b_5, c_5, d_5$ are same both in $J_3$ and $J_5$ in corresponding cycles, we need to change the solution only locally in the vertices with lower value indices $1,2$ and $3$, which actually form the $K_{3,3}^*$ subgraph. It's also easy to check that all $b_i$ vertices are ordered.

\begin{figure*}
  \checkoddpage \ifoddpage \forcerectofloat \else \forceversofloat \fi
  \includegraphics[width=0.95\linewidth]{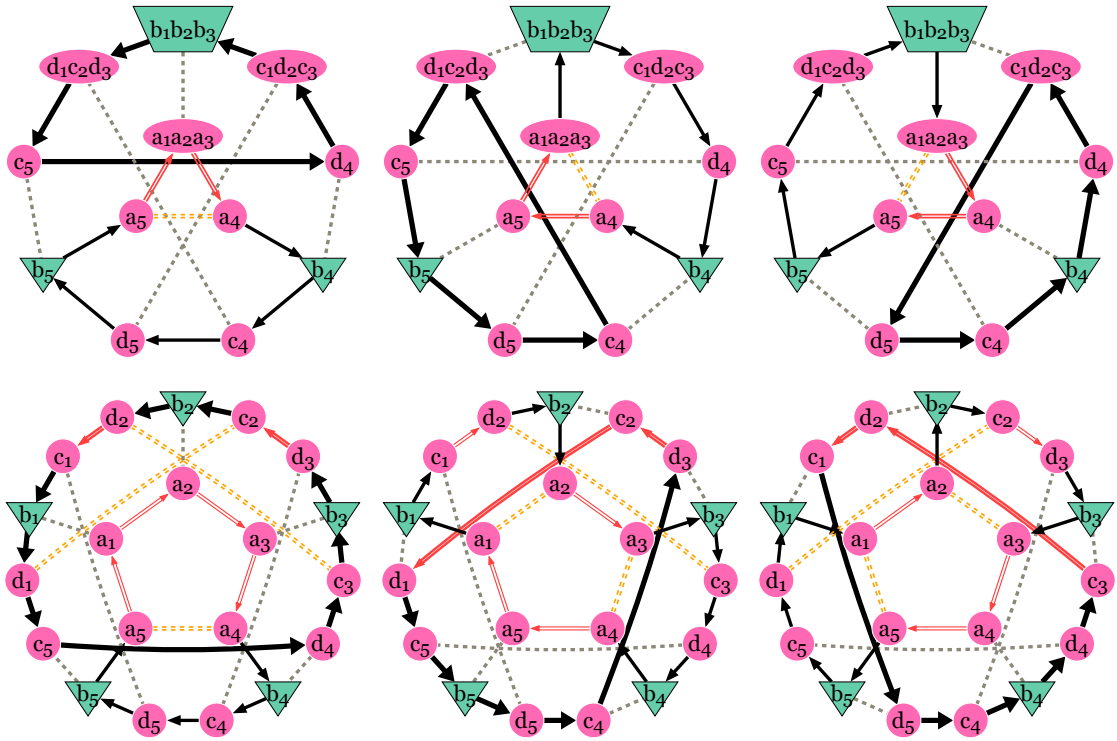}
  \caption{First set of 3 cycles of $J_3$ vs $J_5$}
  \label{fig:j3vsj5part1}
\end{figure*}

\begin{figure*}
  \checkoddpage \ifoddpage \forcerectofloat \else \forceversofloat \fi
  \includegraphics[width=0.95\linewidth]{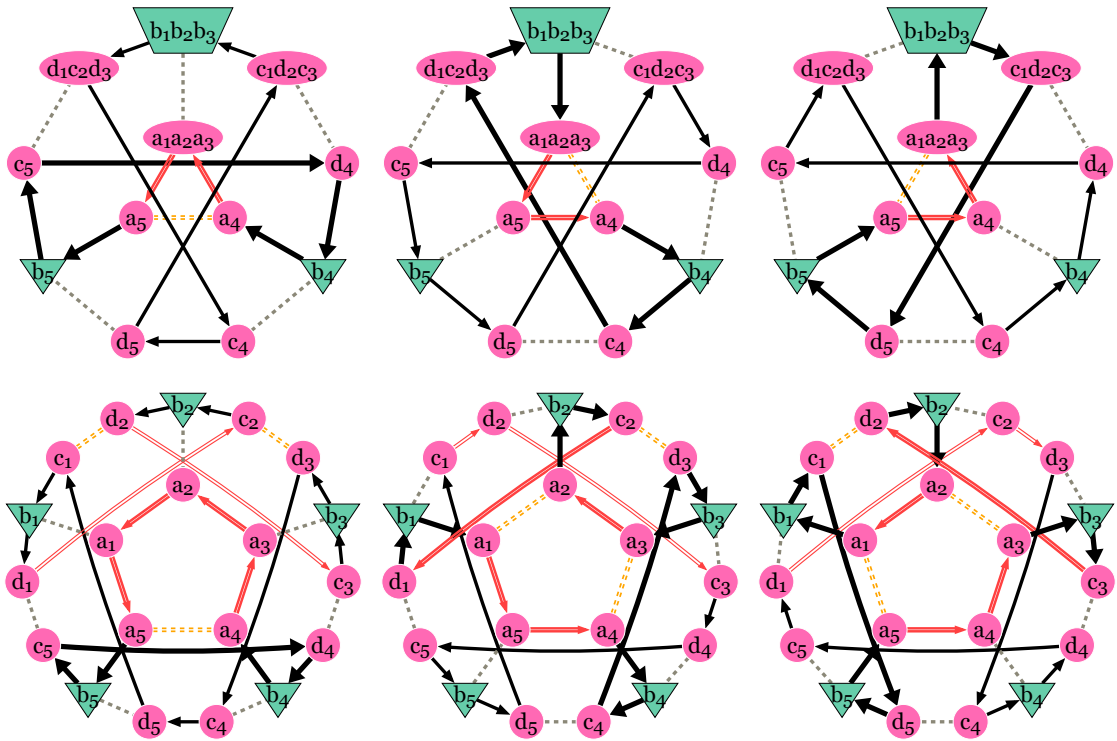}
  \caption{Second set of 3 cycles of $J_3$ vs $J_5$}
  \label{fig:j3vsj5part2}
\end{figure*}

\end{proof}

The solution actually comes from a much more general notion of an \textit{almost strong Petersen colouring}, which we will explore in future paper in the series \bothcite{GP2-2}. Roughly speaking, it is based on the $2$-sphere unit vector flow. While strong Petersen colouring corresponds to a finite subset of 30 points on a sphere, \textit{almost strong Petersen colouring} corresponds to a bigger subset of around 78 points, which still seems to preserve all structures in the oriented way, including o6c4c and o5cdc solutions. It also seems to work more often for snarks, e. g. the number increases from 8 to 1579 snarks, among 3247 snarks with number of vertices $\le 28$, so around $48\%$ have this colouring.

\section{Gluing an orientable surface from o6c4c}

In this section we show how to glue an orientable surface from o6c4c circuits.

But first we need 1 more notion for the edges. For illustration purposes let's check o6c4c for one of Blanuša snarks on 18 vertices, displayed in Fig.~\ref{fig:blanusa}.

\begin{figure*}
  \checkoddpage \ifoddpage \forcerectofloat \else \forceversofloat \fi
  \includegraphics[width=0.93\linewidth]{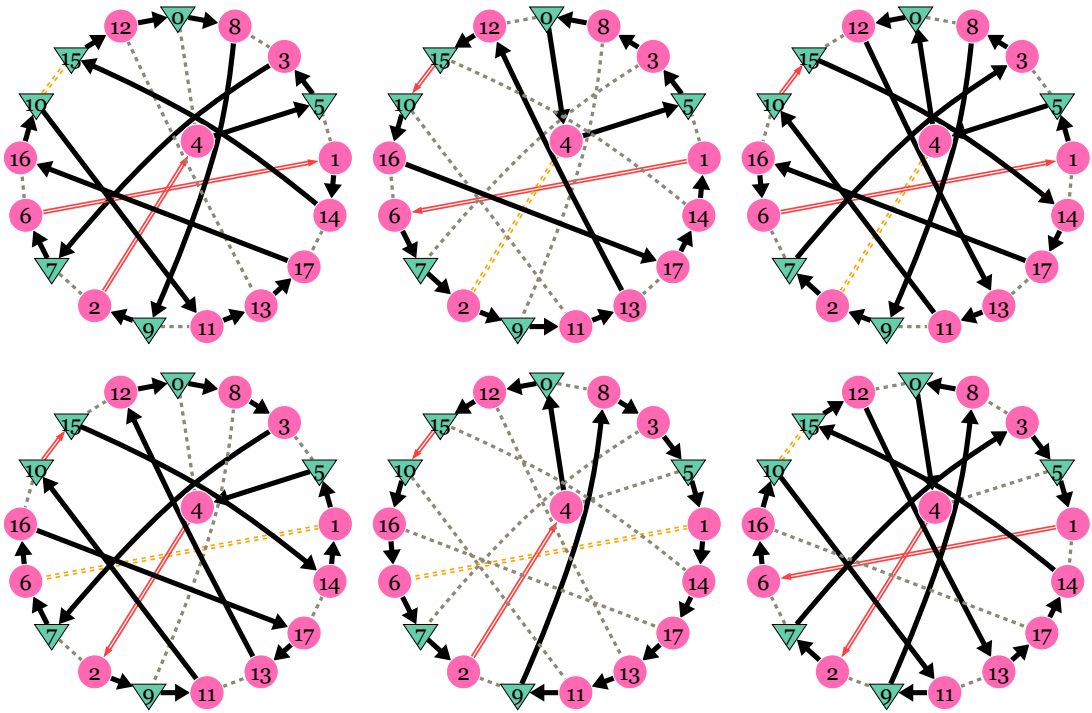}
  \caption{Solution for second Blanuša snark}
  \label{fig:blanusa}
\end{figure*}

We have already divided vertices into ordered and disordered; for 6c4c solutions we can also split edges into 2 classes, \textit{poor} and \textit{rich}. Terminology here is same as in \citep{HagglundSteffen} paper for Petersen colouring situation. The idea is to fix an edge, and look at the 4 appearances of this edge in 6c4c solution, and look at its neighbouring edges. We say that the edge is \textit{poor} if it appears in 2 combinations with its neighbours, and \textit{rich} if it has 4 combinations, see Fig.~\ref{fig:edges}. E. g., for Petersen o6c4c all 15 edges are rich; for the displayed o6c4c for Blanuša snark we have 3 poor edges and 24 rich edges. We also note that if the 6c4c solution comes from Petersen colouring, then the terminology coincides, we get the same sets of rich and poor edges.

\begin{figure}
  \checkoddpage \ifoddpage \forcerectofloat \else \forceversofloat \fi
  \centering
  \subfloat[Rich edge]{\includegraphics[width=0.45\linewidth]{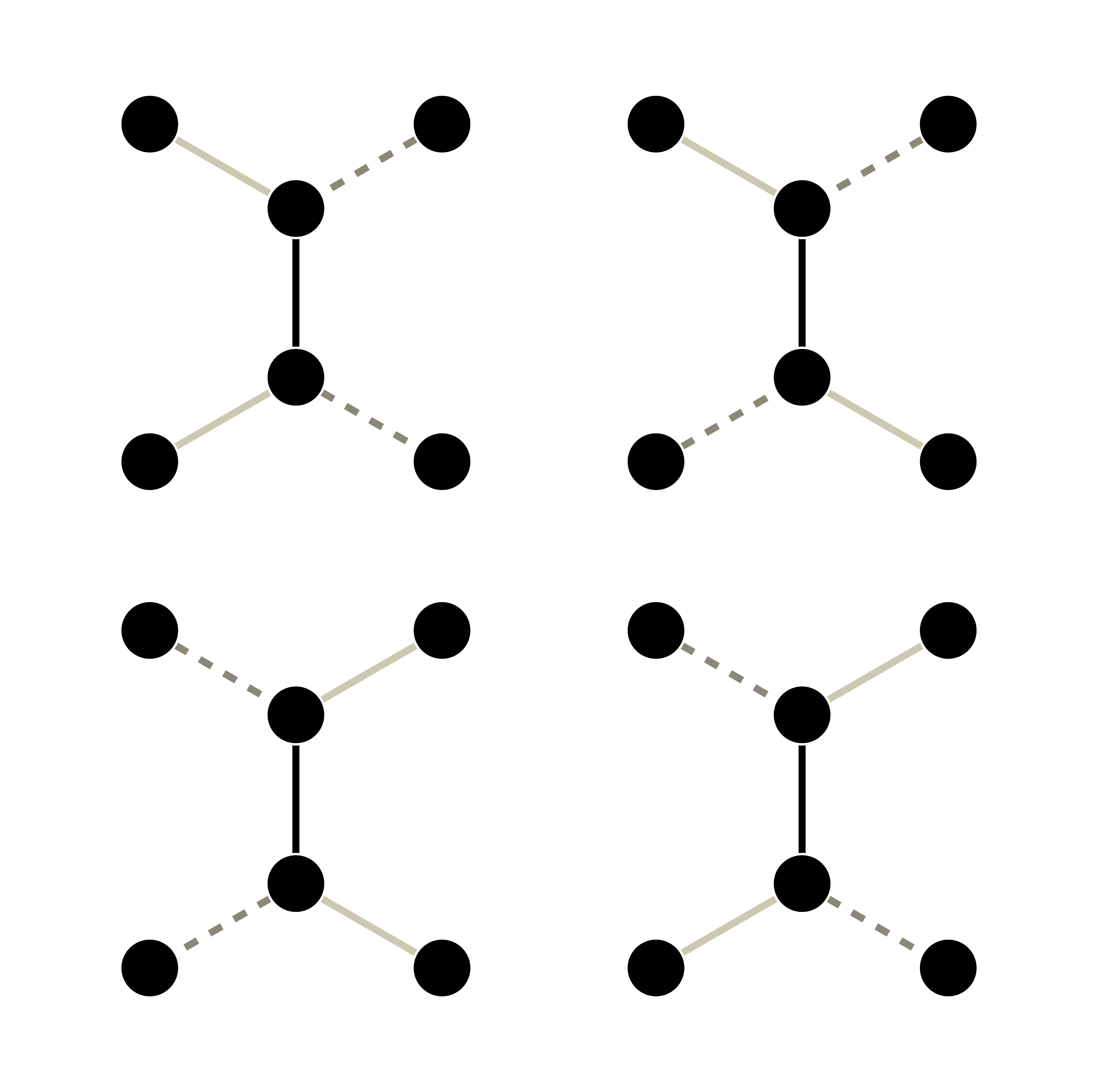}}
  \subfloat[Poor edge]{\includegraphics[width=0.45\linewidth]{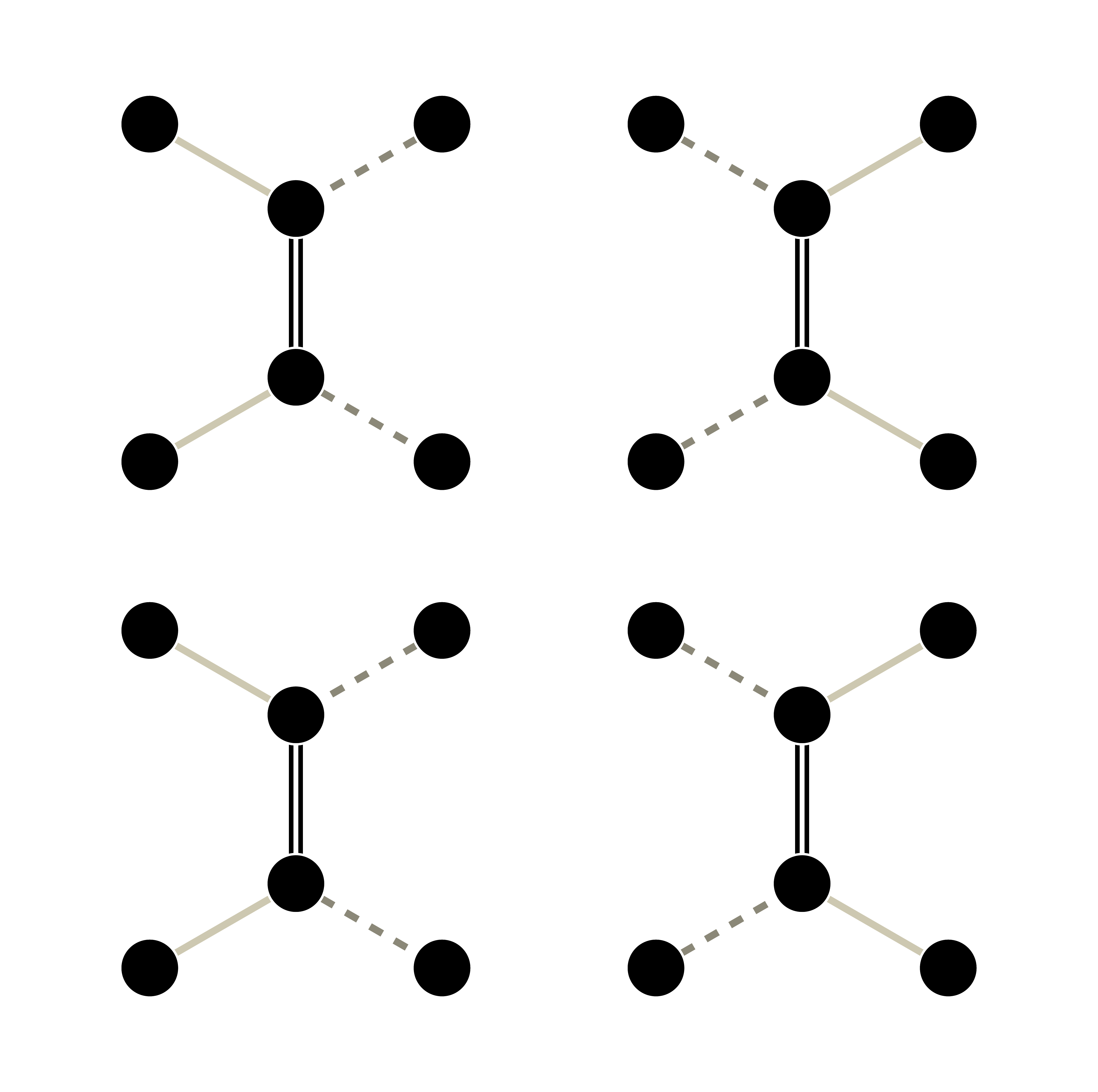}}
  \caption{2 types of edges in 6c4c}
  \label{fig:edges}
  \setfloatalignment{b}
\end{figure}

There's also an interrelationship between vertices and edges in the o6c4c solutions. We have 2 types of rich edges: those that connect disordered and disordered vertices (we call them \textit{drd} edges), and those that connect disordered and ordered vertices (\textit{dro} edges). We also have 2 types of poor edges: those that connect disordered and disordered vertices (\textit{dpd} edges) and those that connect ordered and ordered vertices (\textit{opo} edges).

Now we explain how to glue o6c4c circuits into a surface, in 3 steps:
\begin{enumerate}

\item First we start with an o6c4c without ordered vertices, see Fig.~\ref{fig:all_disordered}.

\begin{figure*}
  \checkoddpage \ifoddpage \forcerectofloat \else \forceversofloat \fi
  \centering
  \includegraphics[width=0.92\linewidth]{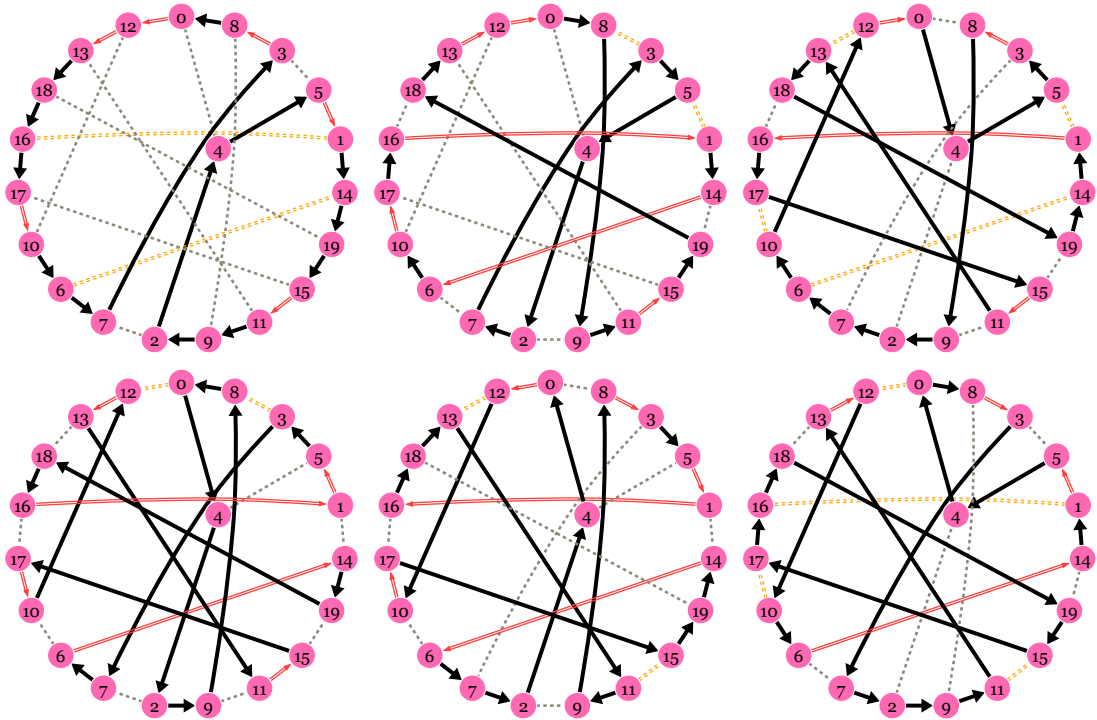}
  \caption{Snark 20.05cyc4-2 without ordered vertices in o6c4c}
  \label{fig:all_disordered}
\end{figure*}

In this case we only have $drd$ and $dpd$ edges. For such an o6c4c we have a nice property, that if we look at any pair of consecutive edges, this pair is covered orientably, meaning it's covered twice and in opposite directions. Then for our graph $G$ we can build a new graph $S$, corresponding to the orientable surface, where edges of $S$ correspond to pairs of consecutive edges of $G$, and vertices of $S$ correspond to edges of $G$ (although not 1-to-1, rich edges of o6c4c of $G$ map to $4$-valent vertices of $S$, while poor edges of o6c4c of $G$ map to two $2$-valent vertices of $S$). It's easy to see that faces of $S$ correspond to circuits in o6c4c. This way we get a surface without a boundary.

\item Let's again return to the Petersen graph. In addition to $drd$ edges it also has $dro$ edge. However we can apply similar gluing procedure ignoring $dro$ edges at first, and by doing so we already will converge on a solution, displayed in Fig.~\ref{fig:petersen_tesselation}. Here we assume that left and right borders are glued together, and same assumption holds for up and down borders. After gluing we get 3 boundaries\sidenote[][]{Although visually they look more like "holes", but they are not holes in the topological sense}, each one corresponds to an ordered vertex in o6c4c solution. 

\begin{figure}
  \checkoddpage \ifoddpage \forcerectofloat \else \forceversofloat \fi
  \centering
  \includegraphics[width=0.75\linewidth]{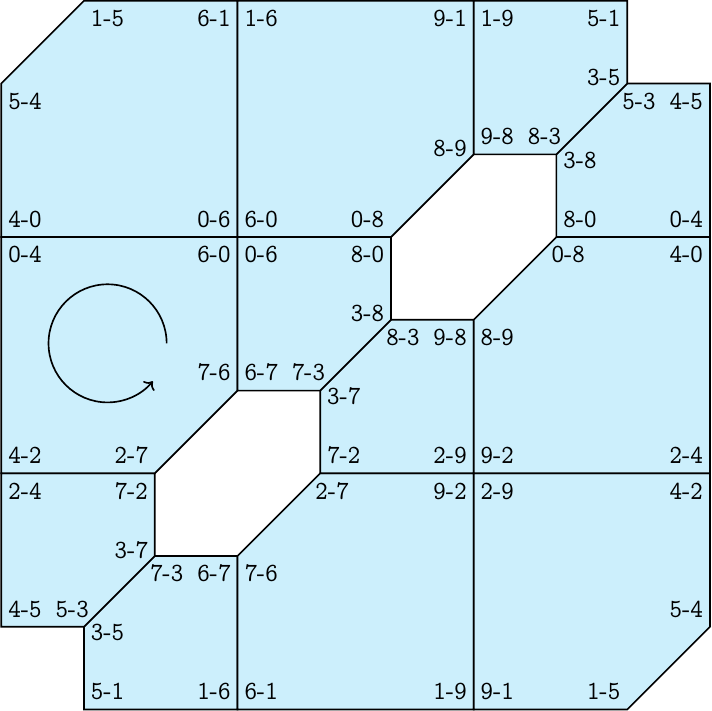}
  \caption{Petersen graph as a surface}
  \label{fig:petersen_tesselation}
  \setfloatalignment{b}
\end{figure}

\item Hopefully these 2 cases are explicit enough to understand what happens in the general case of an o6c4c solution. The general case is quite similar to what happens in Petersen graph case, except for addition of $opo$ edges, so usually boundaries don't have a 1-to-1 correspondence with ordered vertices, but it's something more complicated. We have come up with an \mytextbf{algorithm} for how to find the boundaries. Main idea is that if we know 1 edge of surface graph that lives on a boundary, then it's easy to reconstruct other edges of this boundary, reconstructing them in cyclic order.

To illustrate this algorithm, we will describe an "enter-and-exit the boundary" procedure. E.~g., we can look again at Petersen graph case, say we found surface edge $"(2-7)-(7-6)"$. Then we retrieve the surface face on which it lives, this will be a circuit $"2-7-6-0-4-2"$. We assume that we can enter into the boundary, and exit the boundary, we do this by traversing edges $"(4-2)-(2-7)-(7-6)-(6-0)"$, so $"(4-2)-(2-7)$ is "entering the boundary", and $(7-6)-(6-0)"$ is "exiting" it. Now, we found the exit $drd$ edge (or maybe it will be a $dpd$ edge) and we know how to glue another face to this circuit, this is $"0-6-7-3-8-0"$. We repeat the procedure and find the next surface edge $"(6-7)-(7-3)"$, etc.

\item Basically that's it, but we also wanted to note that the constructed surface can happen to be disconnected, e.~g., if we have a circuit consisting only of poor edges (then we necessarily have another duplicate of this circuit, and they just glue together with each other).

\item As we already mentioned in Section~\ref{pet_o6c4c}, the surface we just built for Petersen graph has already appeared in the literature. The Fig.~\ref{fig:petersen_tesselation} is basically a replica of a figure from \bothcite{DevadossMorava} paper. The first appearance of this surface we know dates back to 1926 \bothcite{BrahanaCoble} paper. Interestingly, there also exists a non-orientable version, if we decide to close the boundaries with cross-caps, see \bothcite{BHV}, and \bothcite{LMY}, which also mentions the relation to Petersen graph. Last but not least, we'd like to mention the great dodecahedron construction from \citep{Apery1998} and \citep{DevadossMorava} which is a dual graph to the orientable double cover of the surface we just built.

\end{enumerate}

\section{A split into 2$\times$6cdcs, and another pair of surfaces}

Sometimes the o6c4c solution splits into two 6-cycle double covers, e.~g. this happens for Petersen graph, each 6cdc containing six 5-circuits. Although these are not orientable cycle double covers, we can still glue an orientable surface with boundary for each of them, see Fig.~\ref{fig:petersen_6cdcs}.

\begin{figure}
  \checkoddpage \ifoddpage \forcerectofloat \else \forceversofloat \fi
  \centering
  \subfloat[Surface 1]{\includegraphics[width=0.45\linewidth]{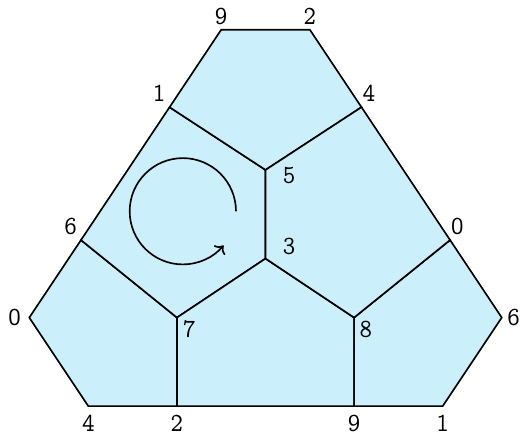}}
  \subfloat[Surface 2]{\includegraphics[width=0.45\linewidth]{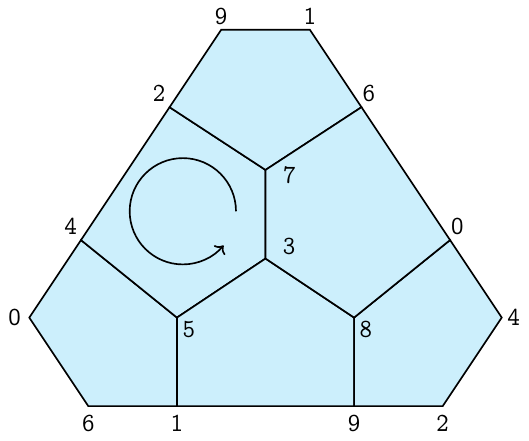}}
  \caption{Petersen o6c4c split into two 6cdcs}
  \label{fig:petersen_6cdcs}
  \setfloatalignment{b}
\end{figure}

Some additional notes:

\begin{itemize}

\item Boundary of each surface consists of the same edges of the graph, and these edges only pass through disordered vertices. It's also easy to check that all \textit{drd} edges are on the boundary, plus some subset of \textit{dpd} edges can also fall on the boundary.

\item In Petersen graph case, each surface is a hemi-dodecahedron, and the boundaries are actually the same, each being a 12-circuit. So we can even glue these 2 surfaces together, and get a graph on a 2-sphere. In some sense this new graph resembles a dodecahedron by having 12 5-circuits, even though it's not. \sidenote[][-0.5cm]{It would become a dodecahedron tessellation if we shift vertices cyclically by 1 in one of the parts.}

\item In general case the boundaries could be different from each other, e.~g. check o6c4c in Fig.~\ref{fig:different_boundaries}. One of the 6cdcs has 1 connected component with 2 boundaries, each of length 10, while the other 6cdc has 3 connected components and 4 boundaries in total, each of length 5.

\begin{figure*}
  \checkoddpage \ifoddpage \forcerectofloat \else \forceversofloat \fi
  \centering
  \includegraphics[width=0.92\linewidth]{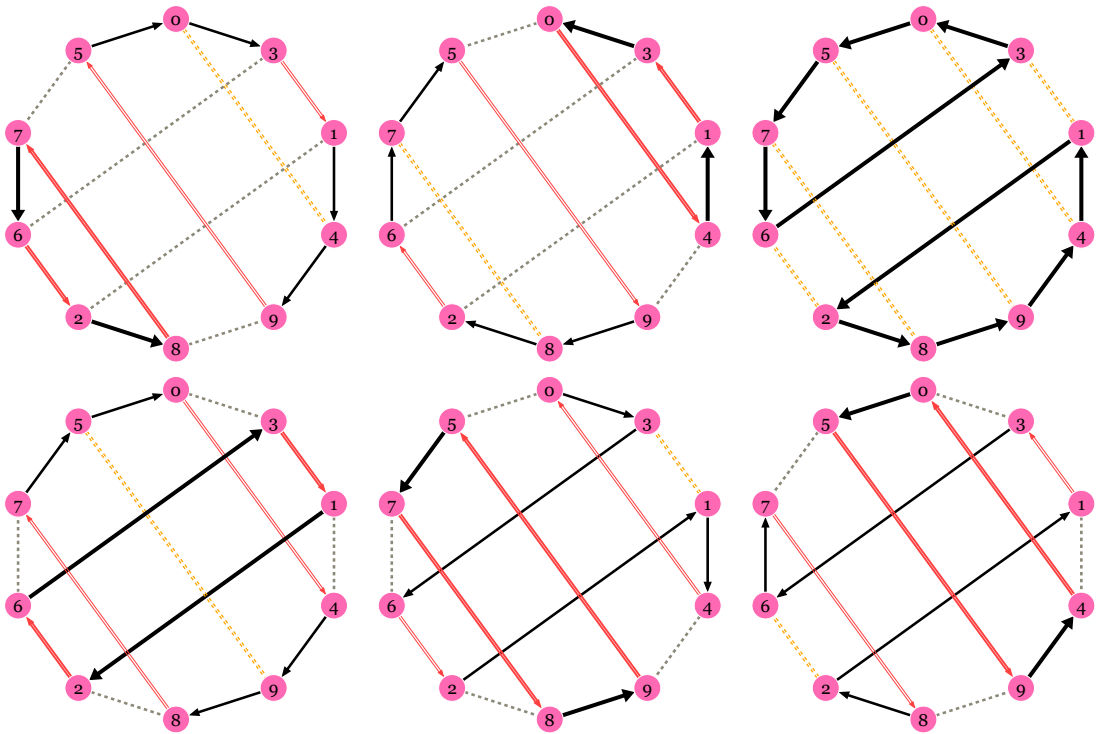}
  \caption{Cubic graph 10.04-6}
  \label{fig:different_boundaries}
\end{figure*}

\end{itemize}

\section{o6cdc from o6c4c}

In this section we will associate a ribbon graph with an o6c4c solution, which also turns up to be an o6cdc solution when all vertices are \textit{disordered}.

\subsection{Ribbon graphs}

\begin{restatable}[ribbon graph]{defn}{rib}
A ribbon graph is a graph with additional cyclic ordering on the half-edges incident to each vertex.
\end{restatable}

A ribbon graph is also another way to represent a (topological) graph embedded into surface.

\subsection{Petersen graph, and o5cdc}

For example, let's again return to the Petersen graph o6c4c depicted in Fig.~\ref{fig:petersen}, and assign the edge orders the following way: half-edges of all vertices except 3 have a counter-clockwise order, while half-edges of vertex 3 are ordered clockwise. E.~g., edges around vertex 0 should be ordered cyclically as $"8 \rightarrow 4 \rightarrow 6 \rightarrow 8"$. The way we get a new (topological) graph embedded into surface is by describing its faces. And to find the faces, we trace their edges. E.~g., let's follow the edge $8 \rightarrow 0$. We use the ribbon graph order, where 4 follows 8, so we deduce that the next edge would be $0 \rightarrow 4$, and so on, until we return to the first edge we started with. The resulting faces are depicted in Fig.~\ref{fig:petersen_o5cdc}.

\begin{figure*}
  \checkoddpage \ifoddpage \forcerectofloat \else \forceversofloat \fi
  \centering
  \includegraphics[width=0.92\linewidth]{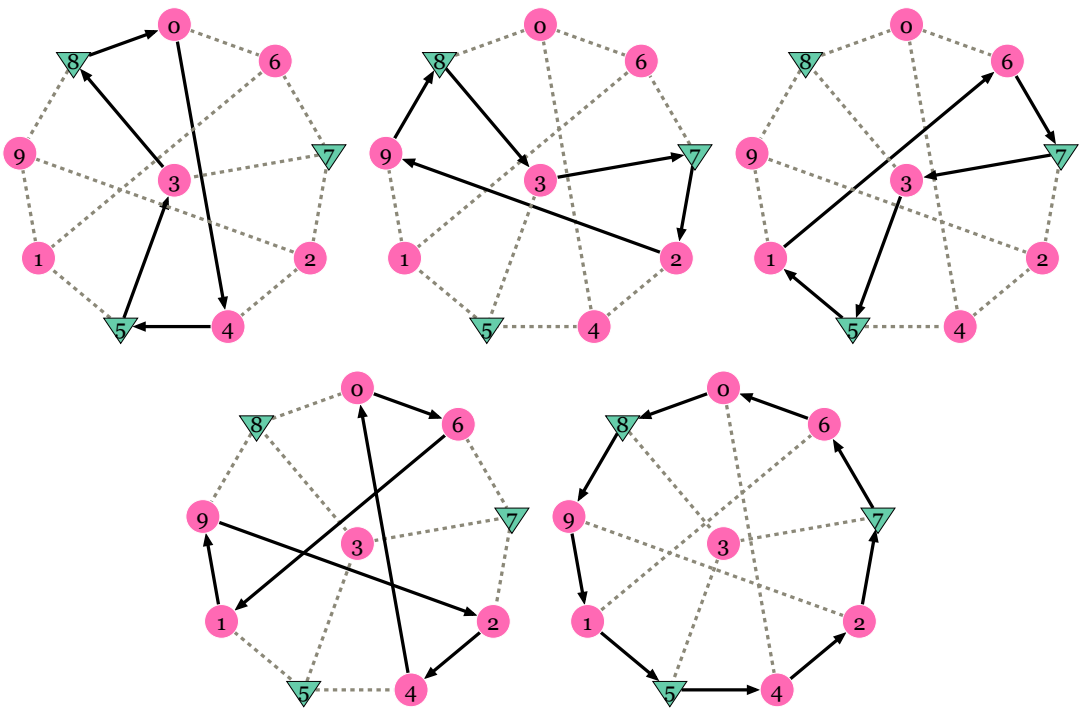}
  \caption{Petersen o5cdc built from o6c4c}
  \label{fig:petersen_o5cdc}
\end{figure*}

By a strange coincidence, the faces we get are actually circuits of the Petersen graph, and what we get is an oriented 5-cycle double cover! For example, if we would have ordered vertex 3 counter-clockwise, we would get a topological face

$$0 \rightarrow 4 \rightarrow 5 \rightarrow 3 \rightarrow 7 \rightarrow 2 \rightarrow 9 \rightarrow 8 \rightarrow 3 \rightarrow 5 \rightarrow 1 \rightarrow 6 \rightarrow 7 \rightarrow 3 \rightarrow 8 \rightarrow 0,$$

which is not a circuit.

Notice also, that we have chosen half-edge orders in such a way, that pairs of edges around the \textit{ordered} vertices 5, 7 and 8 have the same orientations as in the o6c4c solution.\sidenote[][]{This is where the vertices got their name.} We'll show in \citep{GP2-2} that it's related to \textit{almost strong Petersen colouring}.

\subsection{o6cdc for o6c4c without ordered vertices}

It seems that we have chosen the half-edge orders quite randomly in the previous subsection. While this is true, we can behave much more deterministically when we don't have any ordered vertices in the o6c4c, e.~g. as in the Fig.~\ref{fig:all_disordered}. Let's enumerate cycles as $1,2,3$ in the upper row, and $4,5,6$ in the lower row, looking at them from left to right. And let's look at how the edges fall into the perfect matching, in which cycles this happens, and group cycles together for the same edge.\sidenote[][-0.5cm]{This is actually the first time perfect matching is needed in any crucial way in this paper!} E.~g. for vertex 0 we get (unordered) pairs $(1, 2)$, $(3, 5)$ and $(4, 6)$, see Fig.~\ref{fig:all_disordered_v0}.

\begin{figure}
  \checkoddpage \ifoddpage \forcerectofloat \else \forceversofloat \fi
  \centering
  \includegraphics[width=0.75\linewidth]{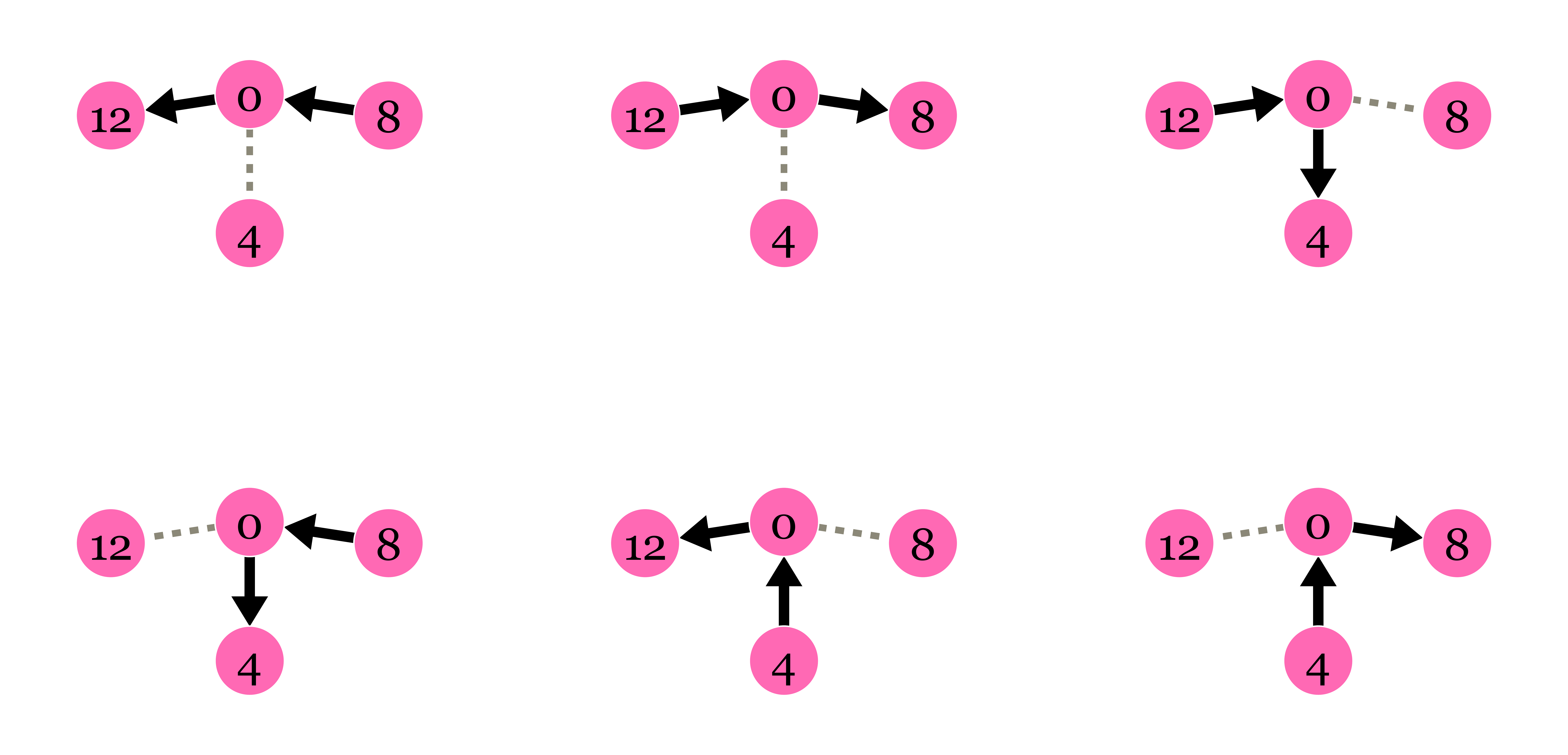}
  \caption{Vertex 0 of the snark 20.05cyc4-2}
  \label{fig:all_disordered_v0}
  \setfloatalignment{b}
\end{figure}

Next step, is we want to decide on the order of numbers inside the pairs, and then we want to decide on the cyclic order between the pairs. For now let's just agree that the pair that contains $1$ has it in the first position. Then we look at cycle 1 and fix the order of edges as "incoming $\rightarrow$ outgoing $\rightarrow$ perfect-matching-edge $\rightarrow$ incoming". Now we can fix the order of numbers in the other pairs as well, so that first number in the pair indicates the cycle with the same order of edges as we fixed. We get the (ordered and cyclically ordered) edge pairs $(1, 2)$, $(3, 5)$ and $(6, 4)$. We will denote this further as $[12 \rightarrow 35 \rightarrow 64]$ for brevity, and it's same as $[35 \rightarrow 64 \rightarrow 12]$ and $[64 \rightarrow 12 \rightarrow 35]$.

Note that we could also put $1$ into a second position in the pair, getting another (cyclic) order $[21 \rightarrow 46 \rightarrow 53]$. However, if we interpret this list of six numbers as a permutation, exactly one of them will produce an \textit{even} permutation.\sidenote[][-0.5cm]{Evenness is defined as a parity of the number of inversions.} In this case it's easy to check, that the cyclic order $[12 \rightarrow 35 \rightarrow 64]$ and all its reorderings correspond to even permutations, which enforces the order of numbers inside the pairs.

So now we can take an o6c4c without ordered vertices and create a ribbon graph deterministically. It turns out, that we can prove the following

\begin{restatable}[o6cdc from o6c4c]{thm}{}
If all vertices in o6c4c solution are disordered, then we get o6cdc from the construction defined above.
\end{restatable}
\begin{proof}
Let's again trace out the edges of one of the faces in the ribbon graph. And let's look locally at some vertex, and write down first the incoming edge, then the outgoing edge, and then separately a perfect-matching edge (let's call such a combination of edges as a \textit{triple}). E.~g., let's say we have a $[12 \rightarrow 35 \rightarrow 64]$ vertex, then we know how the faces that go through this vertex look locally:
\begin{itemize}
\item $(12 \rightarrow 35) + 64$ (meaning $12$ is incoming edge, $35$ is outgoing);
\item $(35 \rightarrow 64) + 12$;
\item $(64 \rightarrow 12) + 35$.
\end{itemize}
Say we have the case $(12 \rightarrow 35) + 64$, then the next pair of edges in the face can be one of the following:
\begin{itemize}
\item $(35 \rightarrow 12) + 64$ (edge $35$ is poor);
\item $(35 \rightarrow 46) + 21$ (edge $35$ is poor);
\item $(53 \rightarrow 16) + 24$ (edge $35$ is rich);
\item $(53 \rightarrow 42) + 61$ (edge $35$ is rich).
\end{itemize}

\vspace{-10px} 

\begin{table}
  \checkoddpage \ifoddpage \forcerectofloat \else \forceversofloat \fi
  \centering
  \begin{tabular}{c@{\hskip 1.5cm}c@{\hskip 1.5cm}c}
\mytextbf{Cycle 1}: & \mytextbf{Cycle 2}: & \mytextbf{Cycle 3} \\
$(12 \rightarrow 43) + 65$ & $(12 \rightarrow 35) + 64$ & $(12 \rightarrow 53) + 46$\\
$(21 \rightarrow 53) + 64$ & $(21 \rightarrow 36) + 54$ & $(21 \rightarrow 63) + 45$\\
$(13 \rightarrow 42) + 56$ & $(13 \rightarrow 24) + 65$ & $(13 \rightarrow 52) + 64$\\
$(31 \rightarrow 62) + 54$ & $(31 \rightarrow 26) + 45$ & $(31 \rightarrow 42) + 65$\\
$(14 \rightarrow 25) + 63$ & $(14 \rightarrow 23) + 56$ & $(14 \rightarrow 62) + 53$\\
$(41 \rightarrow 26) + 53$ & $(41 \rightarrow 25) + 36$ & $(41 \rightarrow 32) + 56$\\
$(15 \rightarrow 24) + 36$ & $(15 \rightarrow 32) + 46$ & $(15 \rightarrow 26) + 43$\\
$(51 \rightarrow 23) + 46$ & $(51 \rightarrow 62) + 43$ & $(51 \rightarrow 24) + 63$\\
$(61 \rightarrow 32) + 45$ & $(16 \rightarrow 42) + 35$ & $(16 \rightarrow 25) + 34$\\
$(16 \rightarrow 52) + 43$ & $(61 \rightarrow 52) + 34$ & $(61 \rightarrow 23) + 54$\\
\\
\mytextbf{Cycle 4}: & \mytextbf{Cycle 5}: & \mytextbf{Cycle 6}:\\
$(12 \rightarrow 54) + 63$ & $(12 \rightarrow 34) + 56$ & $(12 \rightarrow 63) + 54$\\
$(21 \rightarrow 34) + 65$ & $(21 \rightarrow 35) + 46$ & $(21 \rightarrow 43) + 56$\\
$(13 \rightarrow 25) + 46$ & $(13 \rightarrow 26) + 54$ & $(13 \rightarrow 62) + 45$\\
$(31 \rightarrow 24) + 56$ & $(31 \rightarrow 25) + 64$ & $(31 \rightarrow 52) + 46$\\
$(14 \rightarrow 52) + 36$ & $(14 \rightarrow 32) + 65$ & $(14 \rightarrow 26) + 35$\\
$(41 \rightarrow 62) + 35$ & $(41 \rightarrow 52) + 63$ & $(41 \rightarrow 23) + 65$\\
$(15 \rightarrow 23) + 64$ & $(15 \rightarrow 62) + 34$ & $(15 \rightarrow 42) + 63$\\
$(51 \rightarrow 26) + 34$ & $(51 \rightarrow 42) + 36$ & $(51 \rightarrow 32) + 64$\\
$(16 \rightarrow 32) + 54$ & $(16 \rightarrow 23) + 45$ & $(16 \rightarrow 24) + 53$\\
$(61 \rightarrow 42) + 53$ & $(61 \rightarrow 24) + 35$ & $(61 \rightarrow 25) + 43$\\
  \end{tabular}
  \caption{o6cdc triples}
  \label{table:o6cdc_triples}
  \setfloatalignment{b}
\end{table}

\vspace{15px} 

Let's interpret triples from the same face as triples falling into the same equivalence class. Say we have a triple $(ab \rightarrow cd) + ef$. Then we can distill 3 types of "rules" from the observation above:

\vbox{
\begin{enumerate}
\item $(ab \rightarrow cd) + ef$ and $(cd \rightarrow ab) + ef$ are in same class;
\item $(ab \rightarrow cd) + ef$ and $(ab \rightarrow fe) + dc$ are in same class;
\item $(ab \rightarrow cd) + ef$ and $(ba \rightarrow ce) + df$ are in same class;
\end{enumerate}
}

Overall we have 360 such triples to classify. Turns out they fall into 6 equivalence classes, if we follow the rules above.\sidenote[][-0.5cm]{It's simple enough to verify with a bit of coding.} If we look at the ordered pairs, for each triple we can find a specific representative, that has same equivalence class, and where we have shuffled the pairs such that $1$ is in the first pair, $2$ is in the leftmost possible pair, and so on. This way we get 60 representatives. The full list for all faces-cycles can be found in the table~\ref{table:o6cdc_triples} above. Although the table~\ref{table:o6cdc_triples} looks slightly cryptic, it bears a strong resemblance to pentads in the theory of outer automorphism of symmetric group $S_6$ \bothcite{Howard2008}, so it looks plausible to reformulate the proof above reusing this theory as well.

Now it's easy to see that each vertex is present in each face (equivalence class) at most 1 time, so we actually get the circuits of the original graph. We get 6 classes, so the solution we built is an o6cdc.
\end{proof}

\newpage

As an example, for graph in Fig.~\ref{fig:all_disordered} we get the following o6cdc in Fig.~\ref{fig:all_disordered_o6cdc}.

\begin{figure*}
  \checkoddpage \ifoddpage \forcerectofloat \else \forceversofloat \fi
  \centering
  \includegraphics[width=0.92\linewidth]{graph20g2-44-o6cdc.pdf}
  \caption{o6cdc built from o6c4c of the snark 20.05cyc4-2 without ordered vertices}
  \label{fig:all_disordered_o6cdc}
\end{figure*}

\vbox{
Another interesting thing to note is that we can reformulate a classical theorem that 3cdc is equivalent to o4cdc in the terms defined above, because we can get only the following 4 (representative) triples in the ribbon graph:

\begin{itemize}
\item $(14 \rightarrow 25) + 63$;
\item $(41 \rightarrow 25) + 36$;
\item $(14 \rightarrow 52) + 36$;
\item $(41 \rightarrow 52) + 63$.
\end{itemize}
}

\newpage

\section{Miscellaneous}

\subsection{Gallery of o6c4c stuff}

\subsubsection*{nz5 from Petersen graph o6c4c}

\begin{figure}
  \checkoddpage \ifoddpage \forcerectofloat \else \forceversofloat \fi
  \centering
  \includegraphics[width=0.75\linewidth]{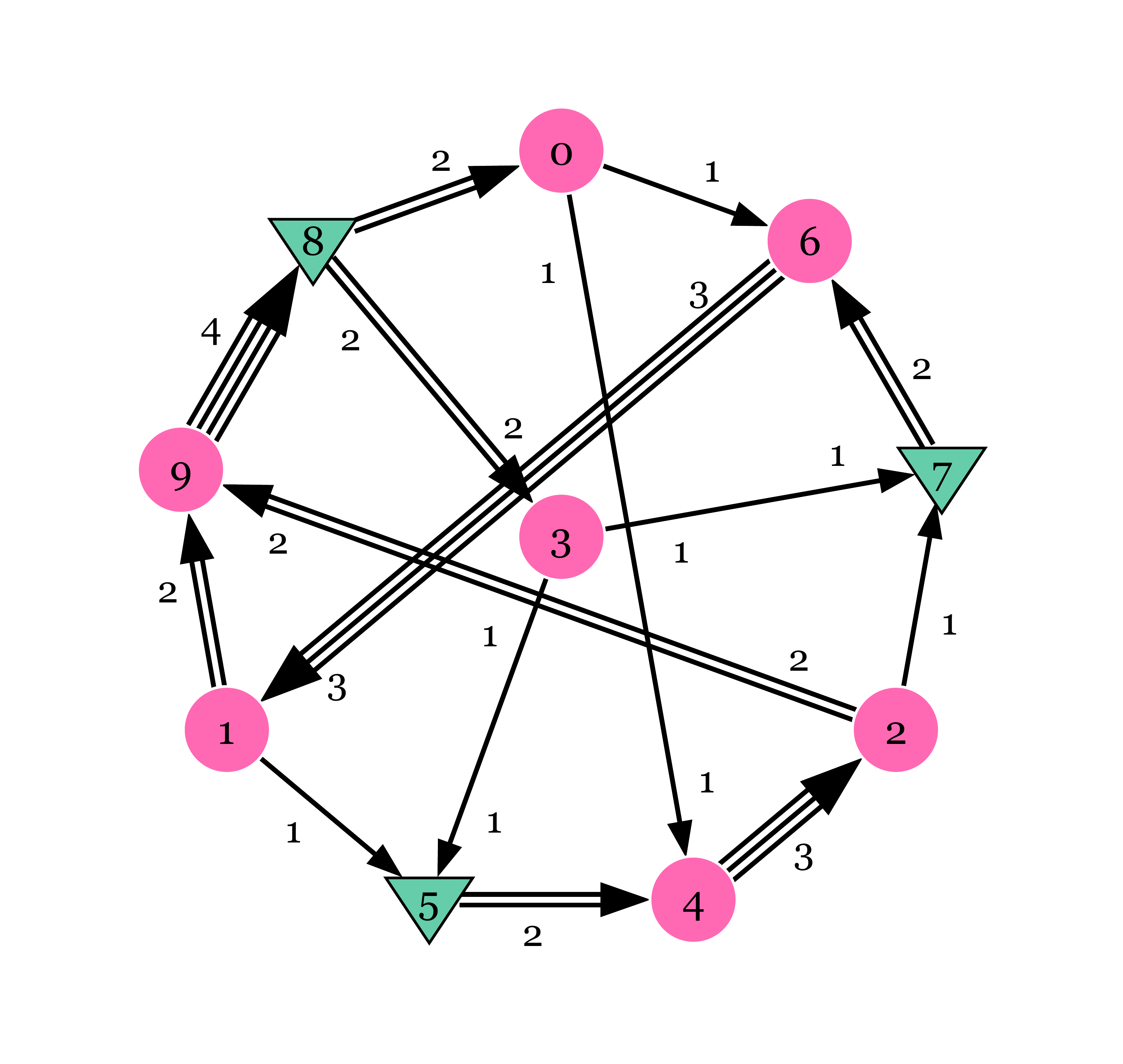}
  \caption{nz5-flow from o6c4c for Petersen graph}
  \label{fig:petersen_nz5}
  \setfloatalignment{b}
\end{figure}

A \textit{nowhere-zero flow} is a non-zero valuation on the edges $\varphi(e), e \in E$ such that the sum of incoming values equals the sum of outgoing values. If all values are non-zero integers such that $\varphi(e) < k$, then we say that it's a \textit{nowhere-zero $k$-flow}. A famous conjecture by W.~T.~Tutte \bothcite{Tutte} states that

\begin{restatable}[nz5]{conj}{}
Every bridgeless cubic graph has a nowhere-zero 5-flow.
\end{restatable}

Turns out it's possible to create a nowhere-zero 5-flow from the o6c4c solution for Petersen graph, by summing up the cycles with weights 0, 1, 2, 0, 2 and 1 (reading cycles in same order as before, first upper row, then lower row, from left to right), see Fig.~\ref{fig:petersen_nz5}. Also interestingly, it's different from a nz5 that we can form from a related o5cdc solutions we built previously.

In general, it seems that the best we can do is build a nz7-flow (by taking any 3 cycles with weights 1, 2 and 4), but this is quite a weak result, because a theorem by P.~Seymour \bothcite{Seymour6flows} states that we can build a nowhere-zero 6-flow for any snark.

\subsubsection*{Snarks, which o6c4c has all edges rich}

There exist other snarks besides Petersen graph, which have all edges rich in o6c4c. One such example is shown in Fig.~\ref{fig:all_rich}.

\begin{figure*}
  \checkoddpage \ifoddpage \forcerectofloat \else \forceversofloat \fi
  \centering
  \includegraphics[width=1.0\linewidth]{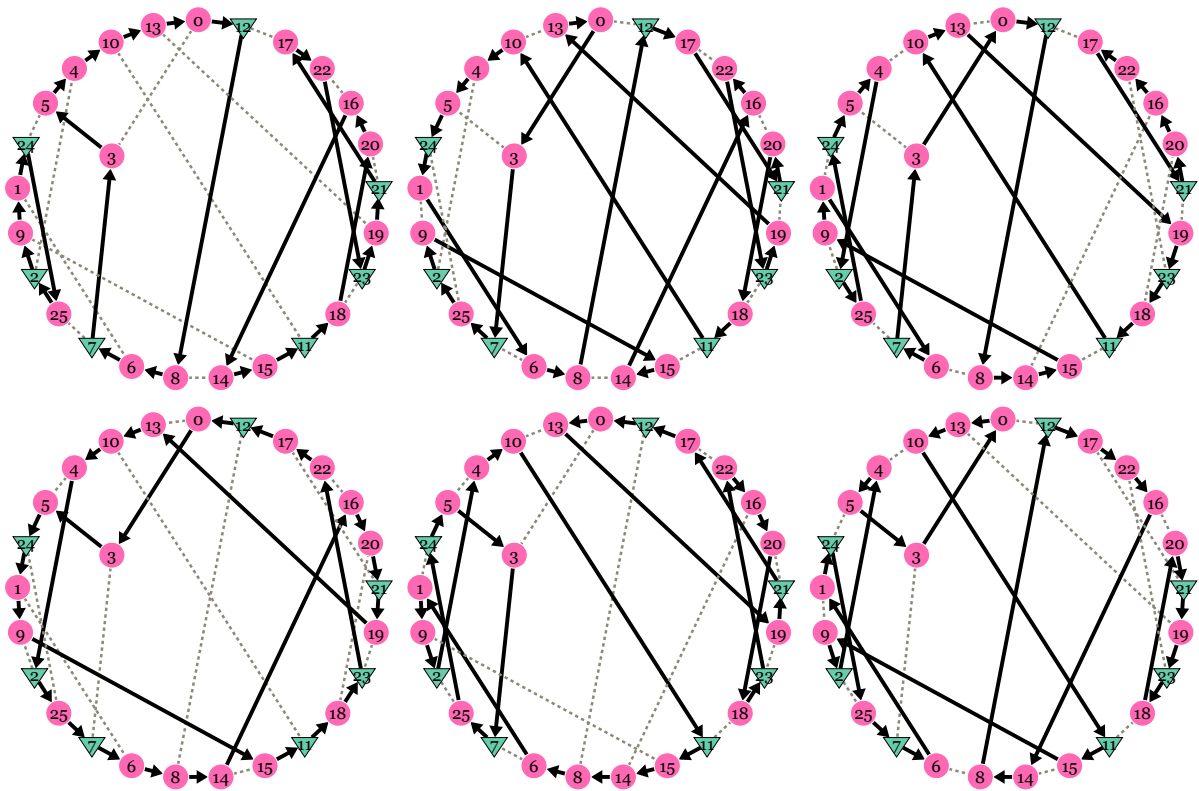}
  \caption{Snark 26.05cyc4-60 without poor edges in o6c4c}
  \label{fig:all_rich}
\end{figure*}

\subsubsection*{Snarks, which o6c4c produces a surface with genus 0}

Although snarks are non-planar, and their ocdc solutions are non-planar either, we can find snarks whose o6c4c surface has genus 0 (after filling up their boundaries), which means that it's possible to draw their o6c4c surface in the plane. The smallest such snarks have 28 vertices, see Fig.~\ref{fig:genus0}.

\begin{figure*}
  \checkoddpage \ifoddpage \forcerectofloat \else \forceversofloat \fi
  \centering
  \includegraphics[width=0.9\linewidth]{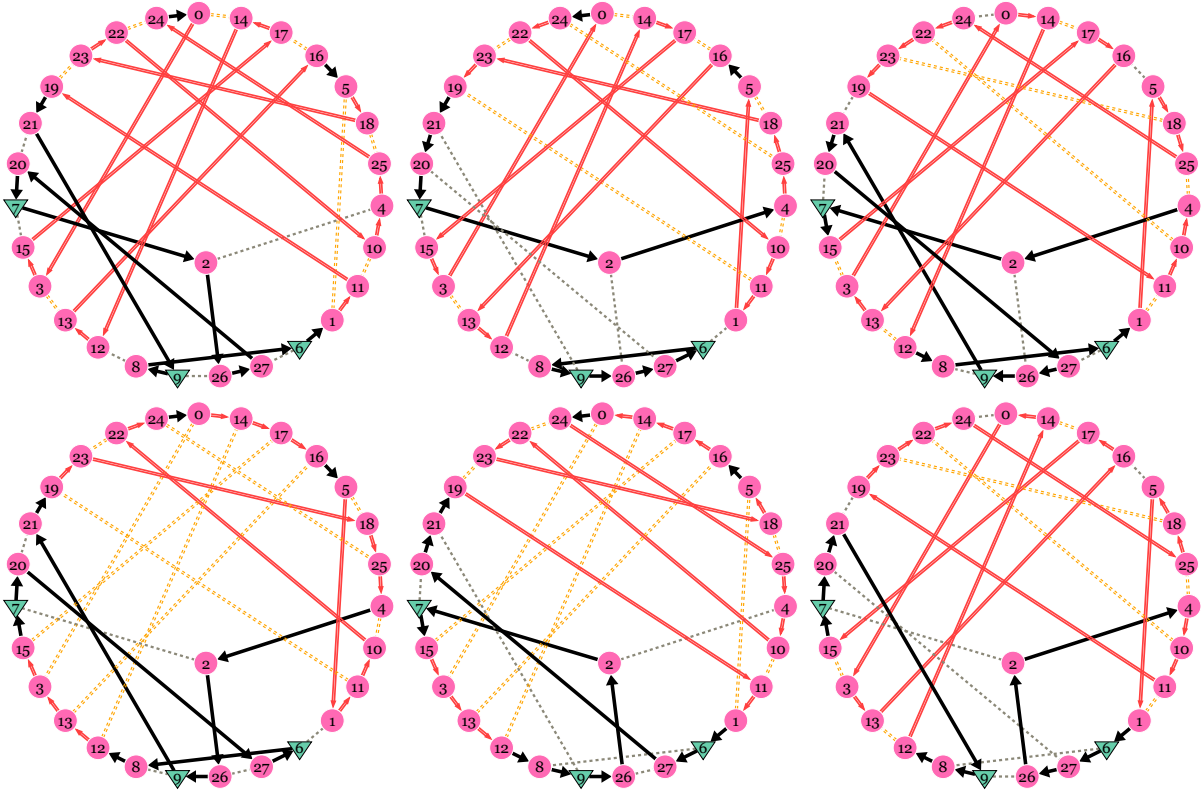}
  \caption{Snark 28.05cyc4-16 whose o6c4c surface has genus~0}
  \label{fig:genus0}
\end{figure*}

\subsection{More observations around o6c4c}

We would like to add some more notes that don't fit the main topic of the paper, but will be explored in \citep{GP1-2}, but we don't have any deadline for the paper, so these notes will live here for quite some time. Various write-ups will be collected in the author's GitHub\sidenote[][-0.5cm]{\url{https://github.com/gexahedron/cycle-double-covers}}.

\begin{restatable}[]{lem}{}
If we have an o6c4c solution, and we get another o6c4c by reorienting some of the circuits, then it has the same parity of the number of ordered vertices.
\end{restatable}
\begin{proof}
Let's compare the two o6c4c solutions. Let's say that each o6c4c consists not of edges, but of \textit{edge-covers}, which correspond to 4-coverings of each edge of the original graph.

For each vertex, let's count how many pairs of edge-covers have changed their orientation. We can get one of $\{0, 2, 3, 4, 6\}$ pair changes, and only 3 corresponds to a switch from ordered to disordered vertex (or vice versa). So let's say we have $n_k$ vertices which had $k$ pair changes.

Now let's count how many edge-covers we need to switch to convert from one o6c4c to another. For each edge, we know that it's covered 4 times, twice in one direction, twice in an opposite direction. The oriented sum of flows in each o6c4c produces a zero flow, so in the 4-covering we can switch 0, 2 or 4 edge-covers between the o6c4c solutions. So this is an even number, let's call it $e_{switched} = 2 m$.

The other way to count this number is to sum up $n_k$ values each with weight $2 k$ (2 edge-covers in a pair), then we get $2 e_{switched}$ (each edge-cover appears in 2 pairs), so,

$$
2 (3n_3 + 2 n_2 + 4 n_4 + 6 n_6) = 2 e_{switched} = 4m,
$$

which implies that $n_3$ is an even number, which is what we wanted.

\end{proof}

\textbf{V3 update}: we now also know how to deduce the parity bit of the number of ordered vertices from non-oriented data, specifically from the number of circuits and the signs of the perfect matchings \bothcite{Norine}, check GitHub notes for more info.

\begin{restatable}[]{lem}{}
Number of ordered vertices~$\ne 1$.
\end{restatable}
\begin{proof}
We can prove this in a slightly indirect way, by using a \textit{(2,4,4)-flows} construction, introduced in paper \bothcite{Hao2009} (and generalized in \bothcite{Chen2015} to non-cubic graphs). The construction says, that if we take any 3 cycles of 6c4c solution, and check the edges that are covered 1 or 3 times, then we get an even cycle (corresponds to a nz2-flow); if we take a subgraph with edges that are covered 0, 1 or 2 times, then we get a nz4-flow subgraph, and same for combination of 0, 2 and 3 times. So what we get is a double cover by 3 subgraphs, having nz2-, nz4- and nz4- flows. And actually vice versa, what they prove is that 6c4c is equivalent to having such a decomposition.

Now, the crucial observation is that this even cycle contains an even number of ordered vertices. If we take the oriented sum of the 3 chosen cycles, and if we try to traverse this even cycle edge by edge, we see that when we pass through disordered vertex, the orientation of edges doesn't change: 1 flows inside, 1 flows outside. If we pass through ordered vertex, we need to switch the orientation of the edges (either both edges flow into the vertex, or they flow out of it).

The last point to make here is that if we could have just 1 ordered vertex in o6c4c, then we could find some triple of cycles (we have 10 different possibilities to form the even cycle, and only 4 of them don't have this vertex), such that this vertex lies on this even cycle from \textit{(2,4,4)-flows}. Which would be a contradiction with that we expect an even number of such vertices on this cycle.
\end{proof}

\subsubsection*{Oriented (2,4,4)-flows}

We know from the papers above that (2,4,4)-flows is equivalent to 6c4c.\sidenote[][]{There is also a related page on the Open Problem Garden website, posted by M.~DeVos \citep{OPG444}, but it is wrong about Petersen graph decomposition.} Two natural questions arise, whether we can orient the (2,4,4)-flows, and whether o6c4c is related to it.

Because nowhere-zero flow introduces some orientation on the edges, we can try to look at the orientations of the edges in the 3 subgraphs that have these nz2-, nz4- and nz4- flows, and check whether each edge is oriented in both ways.

\begin{restatable}[o(2,4,4)-flows]{conj}{}
Every bridgeless (cubic?) graph has an oriented (2,4,4)-flows double cover.
\end{restatable}

The conjecture is trivial for 3-edge-colourable cubic graphs, and we have computationally verified this conjecture for small snarks upto 28 vertices. The solution for Petersen graph is shown in Fig.~\ref{fig:petersen_o244}.

The computational experiments don't show any relation between o(2,4,4)-flows and o6c4c (however it's straightforward to build several nz5-flows from o(2,4,4)-flows, so it's stronger than o6c4c in this sense).

\begin{figure*}
  \checkoddpage \ifoddpage \forcerectofloat \else \forceversofloat \fi
  \centering
  \includegraphics[width=1.0\linewidth]{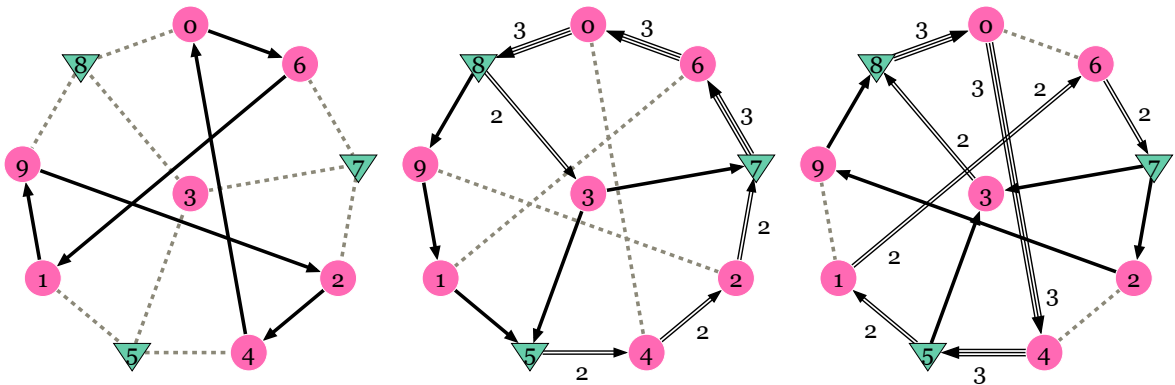}
  \caption{Oriented (2,4,4)-flows for Petersen graph}
  \label{fig:petersen_o244}
\end{figure*}

\subsubsection*{(3,3,3)-flows}

(3,3,3)-flows also seems to be a construction conjecturally related to the \textit{almost strong Petersen colouring}, solution for Petersen graph is shown in Fig.~\ref{fig:petersen_333}. Notice, that ordered vertices correspond precisely to degree 3 vertices in (3,3,3)-flows construction. Not all snarks have an oriented (3,3,3)-flows, but we conjecture that:

\begin{figure*}
  \checkoddpage \ifoddpage \forcerectofloat \else \forceversofloat \fi
  \centering
  \includegraphics[width=1.0\linewidth]{graph10g1-17-333.pdf}
  \caption{(3,3,3)-flows for Petersen graph (actually it's oriented cover)}
  \label{fig:petersen_333}
\end{figure*}

\begin{restatable}[(3,3,3)-flows]{conj}{}
Every bridgeless cubic graph has a (3,3,3)-flows double cover.
\end{restatable}

As in the case of o(2,4,4)-flows, the conjecture is trivial for 3-edge-colourable cubic graphs, and we have computationally verified this conjecture for small snarks upto 28 vertices.

\textbf{V3 update}: we have also studied a stronger notion of $(0,0,1,2,4,8)/3$-nz5-flows\sidenote[][-1.0cm]{Check our slides from Cycles and Colourings 2025 workshop: \url{https://www.researchgate.net/publication/397016716_Slides_Oriented_Berge-Fulkerson_conjecture_and_Unit_vector_flows_on_S2}}, which means that we can sum up the 6 layers as 2-flows, with specified weighting, and get a nowhere-zero 5-flow.

\begin{restatable}[]{thm}{}
o6c4c with $(0,0,1,2,4,8)/3$-nz5-flow implies $(3,3,3)$-flows.
\end{restatable}

\begin{restatable}[]{thm}{}
o6c4c with $(0,0,1,2,-4,-8)/3$-nz5-flow implies $(0,0,1,2,4,8)/3$-nz5-flow (and many other flows as well).
\end{restatable}

\begin{restatable}[]{thm}{}
o6c4c with $(0,0,1,2,-4,-8)/3$-nz5-flow implies $(3,3)$-flow-ppc (see \bothcite{XieCQZhang}), which implies 5cdc.
\end{restatable}

So now we have another relation between o6c4c and (non-oriented) 5cdc, not based directly on Petersen colouring constructions.

\subsubsection*{Euler formula}

Because we have an oriented surface, we can calculate the Euler formula and try to get some useful connections from it.

One interesting case here is when we have a cubic graph $G$ with o6c4c with all disordered vertices, as in Fig.~\ref{fig:all_disordered}.

\begin{restatable}{lem}{}
If we have an o6c4c solution with all disordered vertices, then the number of rich edges has the same parity as the number of circuits in o6c4c.
\end{restatable}
\begin{proof}
Let's apply the Euler formula for surface graph $S$:

$$
V_S - E_S + F_S = 2 - 2 g_S.
$$

It's also easy to show that $V_S$ equals $rich + 2 \cdot poor$. $E_S$ equals the number of pairs of consecutive edges, which equals $3 V_G = 2 E_G$. $F_S$ is just the number of circuits in o6c4c. Finally, if we apply the formula, we will get the statement of the lemma.
\end{proof}

\subsubsection*{Orienting the perfect matching}

We could also try to orient the perfect matching itself, even though we haven't succeeded in this endeavour, to come up with something natural looking. Although, when 6c4c splits into two 6-cycle double covers, and when we don't have perfect matching edges connecting circuits of same 6cdc, one could argue, that we can orient the perfect matching edges from circuits of one 6cdc to circuits of another 6cdc, but currently we don't know how to extract anything useful out of this idea, or how to extend this notion to other 6c4c solutions.

\subsubsection*{More theorems}

\textbf{V3 update}: some of these statements have been previously mentioned as conjectures, but now we have some proofs and write-ups in GitHub.

 \begin{restatable}[]{thm}{}
If we have an o6c4c solution with all disordered vertices, then in each perfect matching we have an even number of rich edges.
\end{restatable}

\begin{restatable}[]{thm}{}
If we have a o6c4c solution which splits into two 6-cycle double covers, then:
\begin{itemize}
\item we have an even number of circuits;
\item we have an even number of \textit{drd} edges (which is same as saying that the sum of numbers of ordered vertices and rich edges is even).
\end{itemize}
\end{restatable}

\subsubsection*{More conjectures}

While doing the computational experiments around o6c4c, we have found much more statements, which hold for small snarks (e.~g. in most cases we checked snarks with at least upto 28 and sometimes 30 vertices), which sound quite easy to state, but which we can't prove at the moment.

\begin{restatable}[]{conj}{}
If we have a (non-oriented!) 6c4c solution which splits into two 6-cycle double covers, then each of the 2-factors has an even number of rich edges.
\end{restatable}

\begin{restatable}[]{conj}{}
If we have an o6c4c which has a "corresponding" $(3,3,3)$-flows (check GitHub notes for more info), then the number of circuits is even.
\end{restatable}

\subsection{6-covers} 

One may also be wondering, what is the situation with 6-covers? To the best of our knowledge, here's the current picture:

\begin{itemize}
\item 3-edge-colourable cubic graphs trivially have a 9c6c and o9c6c.
\item Petersen graph doesn't have a 9c6c solution! This is proven by P.~Seymour in \bothcite{Seymour1979multi}. Originally, when we were thinking about this, we didn't know of the Seymour's proof, and came up with alternative way to prove this fact. Any solution for 9c6c consists of 9 layers, which are complementary to perfect matchings; if we go the complementary problem then it says that we need to cover the graph with 9 perfect matchings which cover the graph 3 times. We have 6 different perfect matchings for Petersen graph, and each edge lies in exactly 2 of them (and each pair of perfect matchings has exactly 1 edge in common). We have at least 2 perfect matchings, each of which we need to use at least 2 times, so we will have some edges which are covered at least 4 times.
\item An easy way to build a 10c6c for Petersen graph is just by taking all of its 9-cycles. Existence of 10c6c is proven for all bridgeless graphs, by G.~Fan in \bothcite[-1cm]{Fan1992}.
\item M.~DeVos and L.~Goddyn have observed that P.~Seymour's 6-flow theorem can be used to construct an o11c6c for every bridgeless graph \citep{OPGMNCC}.
\item o10c6c solutions also exist for Petersen graph, although this particular set of cycles doesn't work (also easy to verify manually), so each o10c6c solution contains at least one 2-factor.
\item o9c6c exists for all small snarks (except for Petersen graph) with upto 26 vertices (and we also checked some bigger snarks, e.~g., treelike/windmill snark on 34 vertices, snarks with circular flow number 5 on 34 vertices, cyclically 5-edge-connected permutation graphs on 34 vertices, oddness 4 snarks with 44 vertices, without finding a counterexample).
\item It also seems possible to introduce notions of \textit{poor} and \textit{rich} edges, as well as \textit{ordered} and \textit{disordered} vertices, but we haven't found anything interesting stuff here for now, as well as we don't know, whether it's possible to glue any kind of surface out of the o9c6c circuits.
\end{itemize}

\bibliography{gp1-1}
\bibliographystyle{alpha}

\end{document}